\documentclass[11pt]{amsart}
\usepackage{color,amssymb}

\usepackage[all]{xy}

\newcommand{\IN}{\mathbb N}
\newcommand{\IQ}{\mathbb Q}
\newcommand{\IC}{\mathbb C}
\newcommand{\IT}{\mathbb T}

\newcommand{\U}{\mathcal U}
\newcommand{\C}{\mathcal C}
\newcommand{\w}{\omega}
\newcommand{\pr}{\mathrm{pr}}

\newcommand{\Ra}{\Rightarrow}

\newcommand{\korin}[2]{\!\sqrt[#1]{#2}}
\usepackage{relsize}
\newcommand*{\defeq}{\stackrel{\mathsmaller{\mathsf{def}}}{=}}

\newtheorem{theorem}{Theorem}[section]
\newtheorem{claim}[theorem]{Claim}

\newtheorem{proposition}[theorem]{Proposition}

\newtheorem{lemma}[theorem]{Lemma}
\newtheorem{problem}[theorem]{Problem}

\theoremstyle{definition}
\newtheorem{remark}[theorem]{Remark}

\newtheorem*{definition*}{Definition}

\title{Subgroups of categorically closed semigroups}
\author{Taras Banakh and Serhii Bardyla}
\address{T.Banakh: Ivan Franko National University of Lviv (Ukraine) and Jan Kochanowski University in Kielce (Poland)}
\email{t.o.banakh@gmail.com}
\address{S.~Bardyla: University of Vienna, Institute of Mathematics, Kurt G\"{o}del Research Center (Austria)}
\thanks{The second author was supported by the Austrian Science Fund FWF (Grant  M 2967).}
\email{sbardyla@yahoo.com}
\makeatletter
\@namedef{subjclassname@2020}{%
  \textup{2020} Mathematics Subject Classification}
\makeatother


\begin{document}
\begin{abstract}Let $\C$ be a class of topological semigroups. A semigroup $X$ is called (1)  {\em $\C$-closed} if $X$ is closed in every topological semigroup $Y\in\C$ containing $X$ as a discrete subsemigroup, 
(2) {\em ideally $\C$-closed} if for any ideal $I$ in $X$ the quotient semigroup $X/I$ is $\C$-closed; (3) {\em absolutely $\C$-closed} if for any homomorphism $h:X\to Y$ to a topological semigroup $Y\in\C$, the image $h[X]$ is closed in $Y$, (4) {\em injectively $\C$-closed} (resp. {\em $\C$-discrete\/}) if for any injective homomorphism $h:X\to Y$ to a topological semigroup $Y\in\C$, the image $h[X]$ is closed (resp. discrete) in $Y$. Let $\mathsf{T_{\!z}S}$ be  the class of Tychonoff  zero-dimensional topological semigroups. We prove the following results: (i) for any ideally $\mathsf{T_{\!z}S}$-closed semigroup $X$, every subgroup of the center $Z(X)=\{z\in X:\forall x\in X\;\;(xz=zx)\}$ is bounded; (ii) for any $\mathsf{T_{\!z}S}$-closed semigroup $X$, every subgroup of the ideal center $I\!Z(X)=\{z\in Z(X):zX\subseteq Z(X)\}$ is bounded; (iii) for any $\mathsf{T_{\!z}S}$-discrete or injectively $\mathsf{T_{\!z}S}$-closed semigroup $X$, every subgroup of $Z(X)$ is finite,  (iv) for any viable idempotent $e$ in an ideally (and absolutely) $\mathsf{T_{\!z}S}$-closed semigroup $X$, the maximal subgroup $H_e$ is ideally (and absolutely) $\mathsf{T_{\!z}S}$-closed and has bounded (and finite) center $Z(H_e)$. 
\end{abstract}

\maketitle

\section{Introduction and Main Results}

In many cases,  completeness properties of various objects of General Topology or  Topological Algebra can be characterized externally as closedness in ambient objects. For example, a metric space $X$ is complete if and only if $X$ is closed in any metric space containing $X$ as a subspace. A uniform space $X$ is complete if and only if $X$ is closed in any uniform space containing $X$ as a uniform subspace. A topological group $G$ is Ra\u\i kov complete  if and only if it is closed in any topological group containing $G$ as a subgroup.

On the other hand, for topological semigroups there are no reasonable notions of (inner) completeness. Nonetheless we can define many completeness properties of semigroups via their closedness in ambient topological semigroups.

A {\em topological semigroup} is a topological space $X$ endowed with a continuous associative binary operation $X\times X\to X$, $(x,y)\mapsto xy$.

\begin{definition*} Let $\C$ be a class of topological semigroups.
A topological
semigroup $X$ is called
\begin{itemize}
\item {\em $\C$-closed} if for any isomorphic topological
embedding $h:X\to Y$ to a topological semigroup $Y\in\C$
the image $h[X]$ is closed in $Y$;
\item {\em injectively $\C$-closed} if for any injective continuous homomorphism $h:X\to Y$ to a topological semigroup $Y\in\C$ the image $h[X]$ is closed in $Y$;
\item {\em absolutely $\C$-closed} if for any continuous homomorphism $h:X\to Y$ to a topological semigroup $Y\in\C$ the image $h[X]$ is closed in $Y$.
\end{itemize}
\end{definition*}

For any topological semigroup we have the implications:
$$\mbox{absolutely $\C$-closed $\Ra$ injectively $\C$-closed $\Ra$ $\C$-closed}.$$

\begin{definition*} A semigroup $X$ is defined to be ({\em injectively, absolutely}) {\em $\C$-closed\/} if so is $X$ endowed with the discrete topology.
\end{definition*}

We will be interested in the (absolute, injective) $\C$-closedness for the classes:
\begin{itemize}
\item $\mathsf{T_{\!1}S}$ of topological semigroups satisfying the separation axiom $T_1$;
\item $\mathsf{T_{\!2}S}$ of Hausdorff topological semigroups;
\item $\mathsf{T_{\!z}S}$ of Tychonoff zero-dimensional topological
semigroups.
\end{itemize}
A topological space satisfies the separation axiom $T_1$ if all its finite subsets are closed.
A topological space is {\em zero-dimensional} if it has a base of
the topology consisting of {\em clopen} (=~closed-and-open) sets. It is well-known (and easy to see) that every zero-dimensional $T_1$ topological space is Tychonoff. 

Since $\mathsf{T_{\!z}S}\subseteq\mathsf{T_{\!2}S}\subseteq\mathsf{T_{\!1}S}$, for every semigroup we have the implications:
$$
\xymatrix{
\mbox{absolutely $\mathsf{T_{\!1}S}$-closed}\ar@{=>}[r]\ar@{=>}[d]&\mbox{absolutely $\mathsf{T_{\!2}S}$-closed}\ar@{=>}[r]\ar@{=>}[d]&\mbox{absolutely $\mathsf{T_{\!z}S}$-closed}\ar@{=>}[d]\\
\mbox{injectively $\mathsf{T_{\!1}S}$-closed}\ar@{=>}[r]\ar@{=>}[d]&\mbox{injectively $\mathsf{T_{\!2}S}$-closed}\ar@{=>}[r]\ar@{=>}[d]&\mbox{injectively $\mathsf{T_{\!z}S}$-closed}\ar@{=>}[d]\\
\mbox{$\mathsf{T_{\!1}S}$-closed}\ar@{=>}[r]&\mbox{$\mathsf{T_{\!2}S}$-closed}\ar@{=>}[r]&\mbox{$\mathsf{T_{\!z}S}$-closed.}
}
$$

$\C$-Closed topological groups for various classes $\C$ were investigated by many authors~\cite{AC1,BL,Ban,BGR,DU,G,L}. In particular, the closedness of commutative topological groups in the class of Hausdorff topological semigroups was investigated in~\cite{Z1,Z2}; $\mathcal{C}$-closed topological semilattices were investigated in~\cite{BBm, BBc, GutikPagonRepovs2010, GutikRepovs2008, Stepp69, Stepp75}. Some notions of completeness in Category Theory were investigated in~\cite{CDT,Cbook1,Er,FG,G1,GH,LW}. In particular, closure operators in different categories were studied in~\cite{BGH, Cbook, CG1, CG2, DT, Dbook, GS, T, Za}.

In some cases the injective $\C$-closedness can be reduced to the $\C$-closedness and $\C$-discreteness.

\begin{definition*}
Let $\C$ be a class of topological semigroups.
A semigroup $X$ is called
\begin{itemize}
\item {\em $\C$-discrete} (or else {\em $\C$-nontopologizable}) if for any injective homomorphism $h:X\to Y$ to a topological semigroup $Y\in\C$ the image $h[X]$ is a discrete subspace of $Y$;
\item {\em $\C$-topologizable} if $X$ is not $\C$-discrete.
\end{itemize}
\end{definition*}

The study of topologizable and nontopologizable semigroups is a classical topic in Topological Algebra that traces its history back to Markov's problem \cite{Markov} of topologizability of infinite groups, which was resolved in \cite{Shelah}, \cite{Hesse} and \cite{Ol} by constructing examples of nontopologizable infinite groups. The topologizability of semigroups was investigated in \cite{BGP,BM,DS1,DS2,DT0,DT1,vD,DI,GLS,KOO,Kotov,Taimanov}.

The following theorem was proved by the authors in \cite[Proposition 3.2]{CCCS}.

\begin{theorem}[Banakh--Bardyla] Let $\C$ be a class of topological semigroups (such that $\C=\mathsf{T_{\!1}S}$). A semigroup $X$ is injectively $\C$-closed if (and only if) $X$ is $\C$-closed and $\C$-discrete.
\end{theorem}

This paper is a continuation of the papers \cite{BB,CCCS,BV} providing inner characterizations of various closedness properties of  semigroups. In order to formulate such inner characterizations, let us recall some properties of semigroups.

A semigroup $X$ is called
\begin{itemize}
\item {\em chain-finite} if any infinite set $I\subseteq X$ contains elements $x,y\in I$ such that $xy\notin\{x,y\}$;
\item {\em singular} if there exists an infinite set $A\subseteq X$ such that $AA$ is a singleton;
\item {\em periodic} if for every $x\in X$ there exists $n\in\IN$ such that $x^n$ is an idempotent;
\item {\em bounded} if there exists $n\in\IN$ such that for every $x\in X$ the $n$-th power $x^n$ is an idempotent;
\item {\em group-finite} if every subgroup of $X$ is finite;
\item {\em group-bounded} if every subgroup of $X$ is bounded.
\end{itemize}

The following theorem (proved in \cite{BB}) characterizes $\C$-closed commutative semigroups.

\begin{theorem}[Banakh--Bardyla]\label{t:C-closed} Let $\C$ be a class of topological semigroups such that $\mathsf{T_{\!z}S}\subseteq\C\subseteq \mathsf{T_{\!1}S}$. A commutative semigroup $X$ is $\C$-closed if and only if $X$ is chain-finite, nonsingular,  periodic, and group-bounded.
\end{theorem}

Let us recall that a {\em congruence} on a semigroup $X$ is an equivalence relation $\approx$ on $X$ such that for any elements $x\approx y$ of $X$ and any $a\in X$ we have $ax\approx ay$ and $xa\approx ya$. For any congruence $\approx$ on a semigroup $X$, the quotient set $X/_\approx$ has a unique semigroup structure such that the quotient map $X\to X/_\approx$ is a semigroup homomorphism. The semigroup $X/_\approx$ is called the {\em quotient semigroup} of $X$ by the congruence $\approx$~.

A subset $I$ of a semigroup $X$ is called an {\em ideal} in $X$ if $IX\cup XI\subseteq  I$. Every ideal $I\subseteq X$ determines the congruence $(I\times I)\cup \{(x,y)\in X\times X:x=y\}$ on $X\times X$. The quotient semigroup of $X$ by this congruence is denoted by $X/I$ and called the {\em quotient semigroup} of $X$ by the ideal $I$. If $I=\emptyset$, then the quotient semigroup $X/\emptyset$ can be identified with the semigroup $X$.

 Theorem~\ref{t:C-closed} implies that each subsemigroup of a $\C$-closed commutative semigroup is $\C$-closed. On the other hand, quotient semigroups of $\C$-closed commutative semigroups are not necessarily $\C$-closed, see Example 1.8 in \cite{BB}. This motivates us to introduce the following notions:

\begin{definition*}A semigroup $X$ is called
\begin{itemize}
\item {\em projectively $\C$-closed} if for any congruence $\approx$ on $X$ the quotient semigroup $X/_{\approx}$ is $\C$-closed;
\item {\em ideally $\C$-closed} if for any ideal $I\subseteq X$ the quotient semigroup $X/I$ is $\C$-closed.
\end{itemize}
\end{definition*}

It is easy to see that for every semigroup the following implications hold:
$$\mbox{absolutely $\C$-closed $\Ra$ projectively $\C$-closed $\Ra$ ideally $\C$-closed $\Ra$ $\C$-closed.}$$

For a semigroup $X$, let $E(X)\defeq\{x\in X:xx=x\}$ be the set of idempotents of $X$.

For an idempotent $e$ of a semigroup $X$, let $H_e$ be the maximal subgroup of $X$ that contains $e$. The union $H(X)=\bigcup_{e\in E(X)}H_e$ of all subgroups of $X$ is called the {\em Clifford part} of $S$.
A semigroup $X$ is called
\begin{itemize}
\item {\em Clifford}  if $X=H(X)$;
\item {\em Clifford+finite} if $X\setminus H(X)$ is finite.
\end{itemize}

Ideally and projectively $\C$-closed commutative semigroups were characterized in  \cite{BB} as follows.

\begin{theorem}[Banakh--Bardyla]\label{t:mainP} Let $\C$ be a class of topological semigroups such that $\mathsf{T_{\!z}S}\subseteq\C\subseteq \mathsf{T_{\!1}S}$. For a commutative semigroup $X$ the following conditions are equivalent:
\begin{enumerate}
\item $X$ is projectively $\C$-closed;
\item $X$ is ideally $\C$-closed;
\item the semigroup $X$ is chain-finite, group-bounded and Clifford+finite.
\end{enumerate}
\end{theorem}

For a semigroup $X$ let
$$Z(X)\defeq\{z\in X:\forall x\in X\;\;(xz=zx)\}$$be the {\em center} of $X$, and
$$I\!Z(X)\defeq\{z\in Z(X):zX\subseteq Z(X)\}$$be the {\em ideal center} of $X$. Observe that the ideal center is the largest ideal in $X$, which is contained in $Z(X)$.

The following theorem is proved in Lemmas 5.1, 5.3, 5.4 of \cite{BB}.

\begin{theorem}[Banakh--Bardyla]\label{t:center} If a semigroup $X$ is  $\mathsf{T_{\!z}S}$-closed, then its center $Z(X)$ is chain-finite, periodic and nonsingular.
\end{theorem}

Theorems~\ref{t:C-closed} and \ref{t:center} suggest the following problem.

\begin{problem}\label{prob:main1} Is the center $Z(X)$ of any $\mathsf{T_{\!z}S}$-closed semigroup $\mathsf{T_{\!z}S}$-closed?
\end{problem}

By Theorems~\ref{t:C-closed} and \ref{t:center},  Problem~\ref{prob:main1} is equivalent to

\begin{problem}\label{prob:main2} Let $X$ be a $\mathsf{T_{\!z}S}$-closed semigroup. Is any subgroup of $Z(X)$ bounded?
\end{problem}

In this paper we shall give partial affirmative answers to Problems~\ref{prob:main1} and \ref{prob:main2}. To formulate them, we need to recall some information on viable idempotents and viable semigroups.

Following Putcha and Weissglass \cite{PW} we define a semigroup $X$ to be {\em viable} if for any $x,y\in X$ with  $\{xy,yx\}\subseteq E(X)$ we have $xy=yx$. This notion can be localized using the notion of a viable idempotent.

An idempotent $e$ in a semigroup $X$ is defined to be {\em viable} if the set
$$\tfrac{H_e}e\defeq\{x\in X:xe=ex\in H_e\}$$
is a {\em coideal} in $X$ in the sense that $X\setminus\frac{H_e}e$ is an ideal in $X$. By $V\!E(X)$ we denote the set of viable idempotents of a semigroup $X$.

By Theorem~3.2 of \cite{BanE}, a semigroup $X$ is viable if and only if each idempotent of $X$ is viable if and only if for every $x,y\in X$ with $xy=e\in E(X)$ we have $xe=ex$ and $ye=ey$. This characterization implies that every semigroup $X$ with $E(X)\subseteq Z(X)$ is viable. In particular, every commutative semigroup is viable.

The main result of this paper is the following theorem describing properties of subgroups of categorically closed semigroups and providing partial affirmative answers to Problem~\ref{prob:main2}.

\begin{theorem}\label{t:main}
Let $X$ be a semigroup and $\mathsf{i}\in\{\mathsf{1,2,z}\}$.
\begin{enumerate}
\item If $X$ is $\mathsf{T_{\!z}S}$-discrete, then $Z(X)$ is group-finite.
\item If $X$ is injectively $\mathsf{T_{\!z}S}$-closed, then $Z(X)$ is group-finite.
\item If $X$ is ideally $\mathsf{T_{\!z}S}$-closed, then $Z(X)$ is group-bounded.
\item If $X$ is $\mathsf{T_{\!z}S}$-closed, then $I\!Z(X)$ is group-bounded.
\item If $X$ is ideally $\mathsf{T_{\!i}S}$-closed, then for every viable idempotent $e\in V\!E(X)$ the maximal subgroup $H_e$ of $X$ is projectively $\mathsf{T_{\!i}S}$-closed and has bounded center $Z(H_e)$.
\item If $X$ is absolutely $\mathsf{T_{\!i}S}$-closed, then for every viable idempotent $e\in V\!E(X)$ the maximal subgroup $H_e$ of $X$ is
absolutely $\mathsf{T_{\!i}S}$-closed and has finite center $Z(H_e)$.
\item If $X$ is ideally $\mathsf{T_{\!z}S}$-closed, then the set $Z(X)\cap\korin{\IN}{V\!E(X)}\setminus H(X)$ is finite.
\end{enumerate}
\end{theorem}
In the last statement of Theorem~\ref{t:main}, $\korin{\IN}{V\!E(X)}$ stands for the set $\{x\in X:\exists n\in\IN\;\;x^n\in V\!E(X)\}$.

Theorem~\ref{t:main}, combined with Theorems~\ref{t:C-closed} and \ref{t:center}, implies the following two partial affirmative answers to Problem~\ref{prob:main1}.

\begin{theorem}
If a semigroup $X$ is $\mathsf{T_{\!z}S}$-closed, then its ideal center $I\!Z(X)$ is $\mathsf{T_{\!1}S}$-closed.
\end{theorem}

\begin{theorem}\label{c:C-closed} If a semigroup $X$ is $\mathsf{T_{\!z}S}$-discrete or injectively $\mathsf{T_{\!z}S}$-closed or ideally $\mathsf{T_{\!z}S}$-closed, then its center $Z(X)$ is $\mathsf{T_{\!1}S}$-closed.
\end{theorem}

Theorems~\ref{t:mainP}, \ref{t:main}(3) and \ref{c:C-closed} motivate the following ``ideal'' version of Problem~\ref{prob:main1}.

\begin{problem}\label{prob:iC1} Is the center of an ideally $\mathsf{T_{\!z}S}$-closed semigroup ideally $\mathsf{T_{\!z}S}$-closed?
\end{problem}

By Theorems~\ref{t:main} and \ref{t:mainP}, Problem~\ref{prob:iC1} is equivalent to

\begin{problem}\label{prob:iC2} Is the center of an ideally $\mathsf{T_{\!z}S}$-closed semigroup Clifford+finite?
\end{problem}

Theorems~\ref{t:mainP} and \ref{t:main}(7) imply the following proposition which gives a partial affirmative answer to Problems~\ref{prob:iC1} and \ref{prob:iC2}.

\begin{proposition}\label{c:iC} If a semigroup $X$ is ideally $\mathsf{T_{\!z}S}$-closed, then any subsemigroup $S\subseteq Z(X)\cap\korin{\IN}{V\!E(X)}$ of $X$ is projectively $\mathsf{T_{\!1}S}$-closed and Clifford+finite.
\end{proposition}

A semigroup $X$ is called {\em $Z$-viable} if $E(X)\cap Z(X)\subseteq V\!E(X)$, i.e., each central idempotent of $X$ is viable. It is clear that each viable semigroup is $Z$-viable.

Proposition~\ref{c:iC} combined with Theorem~\ref{t:center} and Lemma~\ref{l:eiz=>v}  implies the following two  partial answers to Problem~\ref{prob:iC1}.

\begin{theorem} If a $Z$-viable semigroup $X$ is ideally $\mathsf{T_{\!z}S}$-closed, then its center $Z(X)$ is projectively $\mathsf{T_{\!1}S}$-closed.
\end{theorem}

\begin{theorem} If a semigroup $X$ is ideally $\mathsf{T_{\!z}S}$-closed, then its ideal center $I\!Z(X)$ is projectively $\mathsf{T_{\!1}S}$-closed.
\end{theorem}

The statements 1--3, 4, 5--6, and 7 of Theorem~\ref{t:main} are proved in Sections~\ref{s:1-3}, \ref{s:4}, \ref{s:5-6}, \ref{s:7}, respectively.

\section{Preliminaries}

We denote by $\w$ the set of finite ordinals, by $\IN$ the set of positive integers,  and by $\IC$ the set of complex numbers. 
The family of all finite subsets of a set $X$ is denoted by $[X]^{<\w}$.

Let $$\IT\defeq\{z\in\IC:|z|=1\}$$ be the compact topological group endowed with the operation of multiplication of complex numbers. The following lemma is a classical result of Baer \cite[21.1]{Fuchs}.

\begin{lemma}[Baer]\label{l:Baer} For any distinct elements $x,y$ of a commutative group $X$ there exists a homomorphism $h:X\to\IT$ such that $h(x)\ne h(y)$.
\end{lemma}

The following lemma is proved in Claim 7.1 of \cite{CCCS}.

\begin{lemma}\label{l:5.10} For any unbounded commutative group $X$ there exists a homomorphism $h:X\to\IT$ whose image $h[X]$ is infinite.
\end{lemma}

For a semigroup $X$, its
\begin{itemize}
\item  {\em $0$-extension} is the semigroup $X^0=\{0\}\cup X$ where $0\notin X$ is any element such that $0x=0=x0$ for all $x\in X^0$;
\item {\em $1$-extension} is the semigroup $X^1=\{1\}\cup X$ where $1\notin X$ is any element such that $1x=x=x1$ for all $x\in X^1$.
\end{itemize}

If $X$ is a topological semigroup, then we shall assume that $X^0$ and $X^1$ are topological semigroups containing $X$ as a clopen subsemigroup.

For any semigroup $X$, the set $E(X)$ is endowed with the natural partial order $\le$ defined by $x\le y$ iff $xy=yx=x$.
For two idempotents $x,y\in E(X)$ we write $x<y$ if $x\le y$ and $x\ne y$.

For an element $a$ of a semigroup $X$, the set
$$H_a\defeq\{x\in X:(xX^1=aX^1)\;\wedge\;(X^1x=X^1a)\}$$
is called the {\em $\mathcal H$-class} of $a$.
By Corollary 2.2.6 \cite{Howie}, for every idempotent $e\in E(X)$ its $\mathcal H$-class $H_e$ coincides with the maximal subgroup of $X$, containing the idempotent $e$.

\begin{lemma}\label{l:H-commute} Let $X$ be a semigroup, $e,f\in E(X)$, $x\in H_e$ and $y\in H_f$. If $xy=yx$, then $xy\in H_{ef}=H_{fe}$ and $(xy)^{-1}=y^{-1}x^{-1}=x^{-1}y^{-1}$.
\end{lemma}

\begin{proof} Observe that
$$ey=x^{-1}xy=x^{-1}yx=x^{-1}yxe=x^{-1}xye=eye=eyxx^{-1}=exyx^{-1}=xyx^{-1}=yxx^{-1}=ye.$$
By analogy we can prove that $xf=fx$. Next, observe that
$$ef=eyy^{-1}=yey^{-1}=fyey^{-1}=feyy^{-1}=fef=
y^{-1}yef=y^{-1}eyf=y^{-1}ey=y^{-1}ye=fe.$$
Then for the idempotent $u=ef=fe$ we have
$xyX^1=xfX^{1}=fxX^1=feX^1=uX^1$ and $X^1xy=X^1ey=X^1ye=X^1fe=X^1u$,
 which means that $xy\in H_u\subseteq H(X)$.
Observe that $$x^{-1}f=x^{-1}ef=x^{-1}fe=x^{-1}fxx^{-1}=x^{-1}xfx^{-1}=efx^{-1}=fex^{-1}=fx^{-1}.$$
By analogy we can prove that $y^{-1}e=ey^{-1}$.
Then $x^{-1}y^{-1}X^1=x^{-1}fX^1=fx^{-1}X^1=feX^1=uX^1$ and $X^1x^{-1}y^{-1}=X^1ey^{-1}=X^1y^{-1}e=X^1fe=X^1u$, which means that $x^{-1}y^{-1}\in H_u$. By analogy we can prove that $y^{-1}x^{-1}\in H_u$. It follows from $xyy^{-1}x^{-1}=xfx^{-1}=fxx^{-1}=fe=u$ that $y^{-1}x^{-1}=(xy)^{-1}$. Also
$xyx^{-1}y^{-1}=yxx^{-1}y^{-1}=yey^{-1}=eyy^{-1}=ef=u$ implies that $x^{-1}y^{-1}=(xy)^{-1}=y^{-1}x^{-1}$.
\end{proof}

For a subset $A$ of a semigroup $X$, let
$$\korin{\IN}{\!A}\defeq\bigcup_{n\in\IN}\korin{n}{\!A}\quad\mbox{where}\quad \korin{n}{\!A}\defeq\{x\in X:x^n\in A\}.$$
 For an element $a\in X$, the set $\korin{\IN}{\{a\}}$ will be denoted by $\korin{\IN}{a}$.

The following lemma is proved in \cite[Lemma 3.1]{BB}.

\begin{lemma}\label{l:C-ideal} For any idempotent $e$ of a semigroup we have $(\korin{\IN}{H_e}\cdot H_{e})\cup(H_{e}\cdot \korin{\IN}{H_e}\,)\subseteq H_{e}.$
\end{lemma}

\begin{lemma}\label{l:eiz=>v} Let $X$ be a semigroup. Every idempotent $e\in I\!Z(X)$  is viable.
\end{lemma}

\begin{proof} Since $e\in Z(X)$, the map $h:X\to eX$, $h:x\mapsto ex$, is a homomorphism. The semigroup $eX\subseteq Z(X)$ is commutative and hence is viable. Then the set $eX\setminus  \frac{H_e}e=\{x\in eX:ex\notin H_e\}$ is an ideal in $eX$ and its preimage $h^{-1}[eX\setminus \frac{H_e}e]=\{x\in X:eex\notin \frac{H_e}e\}=X\setminus\frac{H_e}e$ is an ideal in $X$. Therefore, $\frac{H_e}e$ is a coideal in $X$ and the idempotent $e$ is viable.
\end{proof}

\begin{remark} The inclusion $E(Z)\cap I\!Z(X)\subseteq V\!E(X)$ proved in Lemma~\ref{l:eiz=>v} cannot be improved to the inclusion $E(X)\cap Z(X)\subseteq V\!E(X)$: by \cite{BG16} and \cite{CM} there exist infinite congruence-free monoids. In each congruence-free monoid $X\ne\{1\}$ the idempotent $1$ is central but not viable.
\end{remark}

\section{Proof of Theorem~\ref{t:main}(1--3)}\label{s:1-3}

First we prove a useful lemma on topologizations of semigroups with the help of uniformities. We refer the reader to \cite[\S8]{Eng} for the theory of uniform spaces. A topology $\tau$ on a group $G$ is called a {\em group topology} on $G$ if $(G,\tau)$ is a Hausdorff topological group.

\begin{lemma}\label{l:uniform} Let $X$ be a semigroup and $H$ be a subgroup of the center $Z(X)$. For any group topology $\tau_{H}$ on $H$ there exists a uniformity $\U$ on $X$ such that the completion $\overline X$ of the uniform space $(X,\U)$ has a unique structure of a topological semigroup containing $X$ as a subsemigroup and $(H,\tau_{H})$ as a topological subgroup.
\end{lemma}

\begin{proof} Let $\U$ be the uniformity on $X$ generated by the base consisting of the entourages
$$W_U=\{(x,y)\in X\times X:x=y\}\cup \bigcup_{z\in X}(Uz\times Uz)$$where $U=U^{-1}\in\tau_H$ is a neighborhood of the idempotent $e$ of the group $H$.

For every neighborhood $U=U^{-1}\in\tau_H$ of $e$ and every $x\in X$, consider the ball $$B(x;W_U)\defeq\{y\in X:(x,y)\in W_U\}$$ of radius $W_U$. We claim that $B(x;W_U)=\{x\}$ for any $x\in X$ such that $x\ne ex$. Indeed, in the opposite case, we could choose an element $y\in B(x;W_U)\setminus\{x\}$ and find $z\in X$ such that $(x,y)\in Uz{\times}Uz$. Then $x=uz$ for some $u\in U\subseteq H$ and hence $ex=euz=uz=x$, which contradicts our assumption.

Fix any $x\in X$ satisfying $x=ex$. Let us show that $B(x;W_U)=UU^{-1}x$ for any symmetric open neighborhood $U\in \tau_H$ of $e$. For any $y\in UU^{-1}x$ we can find $u,v\in U$ such that $y=uv^{-1}x$ and conclude that for $z\defeq v^{-1}x$ we have $x=ex=vv^{-1}x=vz$ and hence $(x,y)=(vz,uz)\in Uz\times Uz\subseteq W_U$, which implies that $y\in  B(x;W_U)$. On the other hand, for every $y\in B(x;W_U)\setminus\{x\}$ there exists $z\in X$ such that $(x,y)\in Uz\times Uz$. It follows that there exists $p\in U$ such that $x=pz$. Then $p^{-1}x=ez$, witnessing that $ez\in U^{-1}x$. Hence $y\in Uz=Uez\subseteq UU^{-1}x$. Thus, $B(x;W_U)=UU^{-1}x$.

Let $\overline X$ be the  completion of the uniform space $(X,\U)$. Let us check that the semigroup operation $\cdot:X\times X\to X$, $\cdot:(x,y)\mapsto xy$, is uniformly continuous with respect to the product uniformity on $X\times X$. Indeed, given any neighborhood $U=U^{-1}\in\tau_H$ of $e$ we can find a neighborhood $V=V^{-1}\in\tau_{H}$ of $e$ such that $VV\subseteq U$. Let us show that
$$B(x;W_V)\cdot B(y;W_V)\subseteq B(xy;W_{VV})\subseteq B(xy;W_U)$$ for every $x,y\in X$. For this fix any $t\in B(x;W_V)$ and $p\in B(y;W_V)$. Assume that $t\neq x$ and $p\neq y$.  Then there exist $z_1,z_2\in X$ and $v_1,v_2,v_3,v_4\in V$ such that $x=v_1z_1$, $t=v_2z_1$, $y=v_3z_2$ and $p=v_4z_2$. Since the group $H$ is contained in the center of $X$ we get that $xy=(v_1v_3)(z_1z_2)$ and $tp=(v_2v_4)(z_1z_2)$. Since $v_1v_3\in VV$ and $v_2v_4\in VV$ we obtain that $tp\in B(xy;W_{VV})\subseteq B(xy;W_U)$, by the choice of $V$. 

The three other cases: $t=x$ and $p\neq y$; $t\neq x$ and $p= y$; $t=x$ and $p=y$ can be treated similarly.

By \cite[8.3.10]{Eng}, the uniformly continuous map $\cdot:X\times X\to X$ can be extended to a uniformly continuous map $\bar\cdot:\overline X \times\overline X\to\overline X$.  The density of $X$ in $\overline X$ implies that the binary operation $\bar\cdot$ is associative, witnessing that $\overline X$ is a topological semigroup. The definition of the uniformity $\U$ ensures that the topology induced by $\U$ on $H$ coincides with the topology $\tau_{H}$.
\end{proof}

The following lemma implies Theorem~\ref{t:main}(1--3).

\begin{lemma}\label{l:groups} Let $X$ be a semigroup.
\begin{enumerate}
\item If $X$ is injectively $\mathsf{T_{\!z}S}$-closed or $\mathsf{T_{\!z}S}$-discrete, then $Z(X)$ is group-finite.
\item If $X$ is ideally  $\mathsf{T_{\!z}S}$-closed, then $Z(X)$ is group-bounded.
\end{enumerate}
\end{lemma}

\begin{proof} If $Z(X)$ is not group-finite (and not group-bounded), then $Z(X)$ contains a countable infinite subgroup $H$ (which is not bounded).

Let $\widehat H$ be the set of all homomorphisms from $H$ to the group $\IT\defeq\{z\in\IC:|z|=1\}$.
By Lemma~\ref{l:Baer}, the diagonal homomorphism $\delta:H\to\IT^{\widehat H}$, $\delta:z\mapsto (\chi(z))_{\chi\in\widehat H}$, is injective. Let $K$ be the closure of the subgroup $\delta[H]$ in the compact topological group $\IT^{\widehat H}$.

Let $\tau_H$ be the unique topology on $H$ such that the injective homomorhism $\delta:H\to K$ is a topological embedding. Let $\U$ be the uniformity on $X$ generated by the base consisting of the entourages
$$W_U=\{(x,y)\in X\times X:x=y\}\cup \bigcup_{z\in X}(U\!z\times U\!z)$$where $U=U^{-1}\in\tau_H$ is a neighborhood of the idempotent $e$ of the group $H$.
In the proof of Lemma~\ref{l:uniform} we have shown that for every $x\in X$ the ball $B(x;W_U)\defeq\{y\in X:(x,y)\in W_U\}$ of radius $W_U$ equals $\{x\}$ if $x\ne ex$ and $B(x;W_U)=UU^{-1}x$ if $x=xe$. Let $\tau_X$ be the topology on $X$ generated by the uniformity $\U$. The topology $\tau_X$ consists of all subsets $V\subseteq X$ such that for every $x\in V$ there exists a neighborhood $U=U^{-1}\in\tau_H$ of $e$ such that $B(x;W_U)\subseteq V$.

By Lemma~\ref{l:uniform}, the completion $\overline X$ of the uniform space $(X,\U)$ carries the structure of a topological semigroup containing $X$ as a dense subsemigroup and $(H,\tau_H)$ as a topological subgroup.  Since $H$ is a clopen  subgroup of $(X,\tau_X)$, the closure $\overline H$ of $H$ in $\overline{X}$ is topologically isomorphic to $K$ and hence is a compact topological group.

Observe that for every $h\in H$ the closed subset $Z_h\defeq \{x\in\overline X:xh=hx\}$ of $\overline X$ contains $X$ and hence $Z_h=\overline X$. Then for every $x\in\overline X$ the closed subset $Z_x\defeq\{y\in \overline X:xy=yx\}$ of $\overline X$ contains $H$ and hence contains the closure of $H$ in $\overline X$. Therefore, $\overline H\subseteq Z(\overline X)$.

Being countable and infinite, the subgroup $\delta[H]$ is not discrete and not closed in the compact topological group $K$ and hence $H$ is not discrete and not closed in $\overline X$. Then there exists an element $a\in\overline H\setminus H\subseteq \overline X\setminus X$. Let $Y$ be the subsemigroup of $\overline X$, generated by the set $X\cup\{a\}$. Since $a\in\overline H\subseteq Z(\overline X)$,  the semigroup $Y$ coincides with the set $\{xa^n:x\in X,\;n\ge 0\}$, where we assume that $xa^n=x$ if $n=0$.

Observe that the space $Y$ is Tychonoff being a subspace of the uniform (and thus Tychonoff) space $\overline X$. For every $x\in X$ with $x\ne xe$, the singleton $\{x\}$ is clopen in $(X,\tau_X)$ and remains clopen in $\overline X$. On the other hand, for every $x=xe\in X$ the subspace $xH$ is clopen in $X$ and $x\overline{H}$ is clopen in $\overline{X}$. Since the set $Y\cap x\overline{H}=\{xya^n:y\in H,\;n\ge 0\}$ is countable, the Tychonoff space $Y$ is locally countable and hence zero-dimensional. Since the subgroup $H$ of $X$ is not discrete, the semigroup $X$ is $\mathsf{T_{\!z}S}$-topologizable. Since the identity homomorphism $X\to Y\in\mathsf{T_{\!z}S}$ is injective and has non-closed image in $Y$, the semigroup $X$ is not injectively $\mathsf{T_{\!z}S}$-closed. These arguments complete the proof of statement (1).
\smallskip

To prove statement (2), assuming that the subgroup $H$ is not bounded, we will show that the semigroup $X$ is not ideally $\mathsf{T_{\!z}S}$-closed. To derive a contradiction, assume that $X$ is ideally $\mathsf{T_{\!z}S}$-closed. By Theorem~\ref{t:center}, the semigroup $Z(X)$ is periodic and so is the group $H\subseteq Z(X)$. By Lemma~\ref{l:5.10}, there exists a homomorphism $h:H\to \IT$ whose image $h[H]$ is infinite and hence dense in $\IT$. Consider the projection $\pr_h:K\to \IT$, $\pr_h:y\mapsto y(h)$, of $K$ onto the $h$-th factor and observe that $\pr_h[K]$ is a compact subgroup of $\IT$ that contains the dense subgroup $h[H]$ of $\IT$. The compactness of $K$ ensures that $\pr_h[K]=\IT$ and hence $K$ contains an element  of infinite order. Since $K$ is topologically isomorphic to the closure $\overline H$ of $H$ in $\overline{X}$, we can find an element $a\in\overline{H}\setminus H\subseteq\overline{X}\setminus X$ of infinite order. Then the monogenic subsemigroup $a^\IN\defeq\{a^n:n\in\IN\}$  of $\overline H$ is disjoint with the periodic group $H$.

In the semigroup $\overline{X}$ consider the ideal $I\defeq\{x\in\overline X:e\notin \overline X^1x\overline X^1\}$. We claim that the ideal $I$ is clopen in $\overline X$. To see that $I$ is closed in $\overline X$, observe that for every $x\notin I$ we have $e=pxq$ for some $p,q\in\overline X^1$. Then for every $h\in \overline H\subseteq Z(\overline X)$ we have $e=ee=pxqe=pxeq=pxhh^{-1}q\in \overline X^1(xh)\overline X^1$ and hence $xh\notin I$. Then $x\overline H$ is a neighborhood of $x$ in $\overline{X}$ that misses the set $I$ and witnesses that the ideal $I$ is closed in $\overline X$. On the other hand, for every $x\in I$ we have $x\overline H\subseteq I$ as $I$ is an ideal in $\overline X$. Since $x\overline H$ is a neighborhood of $x$ in $\overline X$, the ideal $I$ is open in $\overline X$.

In the subsemigroup $Y=\{xa^n:x\in X,\;n\ge 0\}$ of the topological semigroup $\overline X$, consider the set $J\defeq(Y\setminus X)\cup (I\cap Y)$. We claim that  $J$ is an ideal in $Y$.
In the opposite case we can find elements $j\in J$ and $y\in Y$ such that $jy\notin J$ or $yj\notin J$. First assume that $jy\notin J$.  It follows from $j\in J=(Y\setminus X)\cup (I\cap Y)$ and $jy\notin J$ that $jy\notin I$ and $j\notin I$, as $I$ is an ideal. Then $j\in Y\setminus X$ but $jy\in X$. It follows from $j\in Y\setminus X$ that $j=xa^n$ for some $x\in X$ and $n\in \IN$. Also the definition of the semigroup $Y\ni y$ ensures that $y=x_ya^m$ for some $x_y\in X$ and $m\ge 0$ (where we assume that $x_ya^m=x_y$ if $m=0$).

 It follows from $jy\notin I$ that $e=\bar pjy\bar q$ for some $\bar p,\bar q\in \overline X^1$. Since the set $\overline{H}$ is a clopen neighborhood of $e$ in $\overline{X}$, there are neighborhoods $O_{\bar p},O_{\bar q}$ of the elements $\bar p,\bar q$ in the topological semigroup $\overline{X}$ such that $O_{\bar p}jyO_{\bar q}\subseteq \overline H$.  Choose any $p\in X\cap O_{\bar p}$ and $q\in X\cap O_{\bar q}$. Since $jy\in X$, we get $pjyq\in X\cap O_{\bar p}jyO_{\bar q}\subseteq X\cap \overline H=H$. On the other hand, $pjyq=pxa^nx_ya^mq=pxx_yqa^{n+m}$ and hence $pxx_yqe=(pjyq)a^{-(n+m)}\in X\cap\overline H=H$ and finally $pjyq=(pxx_yqe)a^{n+m}\notin H$ because $n+m\in\IN$ and $a^\IN\cap H=\emptyset$. By analogy we can derive a contradiction assuming that $yj\notin J$. These contradictions show that $J$ is an ideal in $Y$.

Let $\tau_Y$ be the topology of the semigroup $Y$ inherited from the topological semigroup $\overline X$. Define a stronger topology $\tau_Y'$ on $Y$ as follows: A subset $U\subseteq Y$ is open in $(Y,\tau_Y')$ if an only if for every $y\in J\cap U$ there exists an open neighborhood $V\in\tau_Y$ of $y$ such that $V\subseteq U$. The definition of $\tau_Y'$ implies that $Y\setminus J=X\setminus I$ is an open discrete subspace of $(Y,\tau_Y')$. Taking into account that $J$ is an ideal in $Y$, it can be shown that $(Y,\tau_Y')$ is a zero-dimensional topological semigroup containing $X$ as a non-closed subsemigroup. Endow the quotient semigroup $Y/(Y\cap I)$ with the strongest topology $\tau$ in which the quotient homomorphism $q:(Y,\tau_Y')\to Y/(Y\cap I)$ is continuous. Taking into account that the ideal $I\cap Y$ is clopen in the zero-dimensional topological semigroup $(Y,\tau_Y')$, we conclude that $(Y/(Y\cap I),\tau)$ is a zero-dimensional topological semigroup containing the quotient semigroup $X/(X\cap I)$ as a discrete non-closed subsemigroup.  Hence the semigroup $X$ is not ideally $\mathsf{T_{\!z}S}$-closed.
\end{proof}

\section{Proof of Theorem~\ref{t:main}(4)}\label{s:4}

Given a $\mathsf{T_{\!z}S}$-closed semigroup $X$, we will prove that its ideal center $I\!Z(X)$ is group-bounded. 

By Theorem~\ref{t:center}, the semigroup $Z(X)$ is chain-finite, periodic and nonsingular. To derive a contradiction, assume that some subgroup of the ideal center $I\!Z(X)$ is unbounded. Observe that for every idempotent $e\in I\!Z(X)$, the maximal subgroup $H_e=H_ee$ of $X$ is contained in $I\!Z(X)$. Therefore, $H_e$ is a maximal subgroup of the semigroup $I\!Z(X)$. Since the semigroup $Z(X)$ is chain-finite, the partially ordered set $E(X)\cap Z(X)$ is well-founded, i.e., each nonempty subset of $E(X)\cap Z(X)$ contains a minimal element. 
Using this fact we can find an idempotent $e\in I\!Z(X)$ such that the maximal subgroup $H_e$ is unbounded but for any idempotent $f\in E(I\!Z(X))$ with $f<e$ the maximal subgroup $H_f$ is bounded.

By Lemma~\ref{l:eiz=>v}, the idempotent $e$ is viable and hence the set $X\setminus\frac{H_e}e$ is an ideal in $X$.

\begin{claim}\label{cl:quot} For any $a\in X\setminus\frac{H_e}e$, the set $G_a=\{x\in H_e: ax=ae\}$ is a subgroup of $H_e$ such that the quotient group $H_e/G_a$ is bounded.
\end{claim}

\begin{proof} Observe that for any $x,y\in G_a$, $axy=aey=aye=aee=ae$. Hence   $G_a$ is a subsemigroup of $H_e$. Since the group $H_e\subseteq Z(X)$ is periodic, the subsemigroup $G_a$ of $H_e$ is a subgroup of $H_e$. It remains to prove that the quotient group $H_e/G_a$ is bounded. Since $e\in I\!Z(X)$, the element $ae$ belongs to the ideal center $I\!Z(X)$. Since $Z(X)$ is periodic, there exists $n\in\IN$ such that $f\defeq(ae)^n=a^ne$ is an idempotent of the semigroup $I\!Z(X)$. It is clear that $ef=fe=f$ and hence $f\le e$. Assuming that $f=e$, we conclude that $ae\in \korin{\IN}{\!H_f}=\korin{\IN}{\!H_e}$ and hence $ae=aee\in\korin{\IN}{\!H_e}H_e\subseteq H_e$ by Lemma~\ref{l:C-ideal}. But the inclusion $ae\in H_e$ contradicts $a\notin\frac{H_e}e$.
This contradiction shows that $f<e$. By the minimality of $e$, the maximal group $H_f$ is bounded. So, there exists $p\in\IN$ such that $y^p=f$ for all $y\in H_f$.

In the group $H_e$ consider the subgroup $G=\{x^p:x\in H_e\}$. By Lemma~\ref{l:H-commute}, $fH_e\subseteq H_f$. For every $x\in H_e$ we have $fx^p=(fx)^p=f$ by the choice of $p$. Then $fG=\{f\}$ and $yG=(yf)G=y(fG)=yf=y$ for every $y\in H_f$. Since $a^ne=(ae)^n\in H_f$, we have $a^nG=a^n(eG)=(a^ne)G=\{a^ne\}$ and hence $G\subseteq G_{a^n}\defeq\{x\in H_e:a^nx=a^ne\}$.

Let $k\le n$ be the smallest number such that the subgroup $G\cap G_{a^k}$ has finite index in $G$.

If $k>1$, then the subgroup $G\cap  G_{a^{k-1}}$ has infinite index in $G$, by the minimality of $k$. Since the group $G\cap G_{a^k}$ has finite index in $G$, the subgroup $G\cap G_{a^{k-1}}$ has infinite index in the group $G\cap G_{a^k}$. So, we can find an infinite set $J\subseteq G\cap G_{a^k}$ such that $x(G\cap G_{a^{k-1}})\cap y(G\cap G_{a^{k-1}})=\emptyset$ for any distinct elements $x, y\in J$.
Observe that for any distinct elements $x, y\in J$ we have $a^kx =a^ke =a^ky$ and $a^{k-1}x\ne a^{k-1}y$ (assuming that $a^{k-1}x =a^{k-1}y$, we obtain that $a^{k-1}e =a^{k-1}xx^{-1}=a^{k-1}yx^{-1}$ and hence $yx^{-1}\in G\cap G_{a^{k-1}}$ which contradicts the choice of the set $J$).
Then the set $A=a^{k-1}J\subseteq I\!Z(X)$ is infinite. We claim that $AA$ is a singleton. Indeed, for any $x, y\in J$ we have $$a^{k-1}xa^{k-1}y=a^kxa^{k-2}y=a^kea^{k-2}y=a^{k-2}ea^ky=e^{k-2}ea^ke =a^{2k-2}e.$$ Therefore, $AA =\{a^{2k-2}e\}$. But the existence of such set $A\subseteq I\!Z(X)\subseteq Z(X)$ contradicts the nonsingularity of the semigroup $Z(X)$. This contradiction shows that $k=1$ and hence the subgroup $G\cap G_a$ has finite index in $G$. Then the quotient group $G/(G\cap G_a)$ is finite and bounded. Since the quotient group $H_e/G$ is bounded, the quotient group $H_e/(G \cap G_a)$ is bounded and so is the quotient group $H_e/G_a$.
\end{proof}

Let $$\korin{\IN}{1}\defeq\{z\in\IC:\exists n\in\IN\;\;(z^n=1)\}$$ be the quasi-cyclic group, considered as a dense subgroup of the compact topological group $\IT=\{z\in\IC:|z| =1\}$. Denote by $\widehat{H}_e$ the set of all homomorphisms from $H_e$ to $\IT$. Since $H_e\subseteq I\!Z(X)\subseteq Z(X)$ and the semigroup $Z(X)$ is periodic, the group $H_e$ is  periodic and hence $\varphi[H_e]\subseteq\korin{\IN}{1}$ for any homomorphism $\varphi\in\widehat{H}_e$. By Lemma~\ref{l:Baer}, the
diagonal homomorphism $\delta:H_e\to\IT^{\widehat{H}_e}$,
$\delta:x\mapsto (\varphi(x))_{\varphi\in\widehat{H}_e}$, is
injective.
Identify the group $H_e$ with its image $\delta[H_e]\subseteq
\korin{\IN}{1}^{\widehat{H}_e}$ in the compact topological group $\IT^{\widehat{H}_e}$ and
let $\overline H_e$ be the closure of $H_e$ in $\IT^{\widehat{H}_e}$.

By Lemma~\ref{l:5.10}, there exists a homomorphism
$h:H_e\to\IT$ with infinite image $h[H_e]$. The subgroup
$h[H_e]$, being infinite, is dense in $\IT$. The homomorphism $h$
admits a continuous extension $\bar h:\overline H_e\to\IT$, $\bar
h:(z_\varphi)_{\varphi\in\widehat{H}_e}\mapsto z_h$.
The compactness of $\overline H_e$ and density of $h[H_e]=\bar h[
H_e]$ in $\IT$ imply that $\bar h[\overline H_e]=\IT$.

By Claim~\ref{cl:quot}, for every $a\in X\setminus\frac{H_e}e$ the quotient group $H_e/G_a$ is bounded. So, we can find a
number $n_a\in\IN$ such that $x^{n_a}\in G_a$ for all $x\in H_e$.  Moreover, for any
non-empty finite set $F\subseteq X\setminus\frac{H_e}e$ and the number
$n_F=\prod_{a\in
F}n_a\in\IN$, the intersection $G_F=\bigcap_{a\in F}G_a$ contains
the $n_F$-th power $x^{n_F}$ of any element $x\in H_e$.

Then for every $y\in h[H_e]\subseteq \korin{\IN}{1}$, we get $y^{n_F}\in
h[G_{F}]$, which implies that the subgroup $h[G_{F}]$ is dense in
$\IT$. Let $\overline G_F$ be the closure of $G_F$ in the compact
topological group $\overline H_e$.  The density of the subgroup
$h[G_F]$
in $\IT$ implies that $\bar h[\overline G_F]=\overline{h[G_F]}=\IT$.

 By the compactness, $\bar h[\bigcap_{F\in [X\setminus\frac{H_e}e]^{<\w}}\overline
G_{F}]=\bigcap_{F\in [X\setminus\frac{H_e}e]^{<\w}}\bar h[\overline G_F]=\IT$. So, we can
fix an element $s\in \bigcap_{F\in [X\setminus\frac{H_e}e]^{<\w}}\overline G_F\subseteq\overline
H_e$ whose image $\bar h(s)\in\IT$ has infinite order in the group
$\IT$. Then $s$ also has infinite order and its orbit
$s^\IN$ is disjoint with the periodic group $H_e$.

Consider the subsemigroup $S\subseteq\overline H_e$ generated by
$H_e\cup\{s\}$. Observe that $$\textstyle S=H_e\cup\{gs^n:g\in H_e,\;n\in\IN\}\subseteq
\prod_{\varphi\in\widehat{H}_e}\IQ_\varphi$$ where $\IQ_\varphi$ is the
countable subgroup of $\IT$ generated by the set
$\korin{\IN}{1}\cup\{\varphi(s)\}$.

It is clear that the subspace topology $\tilde\tau$ on $S$, inherited from the topological group
$\prod_{\varphi\in\widehat{H}_e}\IQ_\varphi$ is Tychonoff and
zero-dimensional.
Then the topology $\tau'$ on $S$ generated by the base
$$\big\{U\cap a\overline G_F:U\in\tilde\tau,\;a\in
\overline H_e,
\;F\in\big[X{\setminus}\tfrac{H_e}e\big]^{<\w}\big\}$$ is zero-dimensional, too.
It is easy to see that $(S,\tau')$ is a
topological semigroup and $s$  belongs to the closure of $H_e$ in
the topology $\tau'$. Finally, endow $S$ with the
topology
$\tau=\{U\cup D:U\in\tau',\;D\subseteq H_e\}$. The topology $\tau$
is
well-known in General Topology as the Michael modification of the
topology $\tau'$ (see \cite[5.1.22]{Eng}). Since the (group)
topology $\tau'$ is zero-dimensional, so is its Michael
modification $\tau$ (see \cite[5.1.22]{Eng}). Using the fact that
$S\setminus H_e$ is an ideal in $S$, it can be
shown that $(S,\tau)$ is a zero-dimensional  topological
semigroup, containing $H_e$ as a dense discrete subgroup. From now
on we consider $S$ as a topological semigroup, endowed
with the topology $\tau$.

 Let $Y=S\sqcup
(X\setminus H_e)$ be the topological sum of the topological space $S$ and the
discrete topological space $X\setminus H_e$. It is clear that
$Y$ contains $X$ as a proper dense discrete subspace.

It remains to extend the semigroup operation of $X$ to a
continuous
commutative semigroup operation on $Y$. In fact, for any
$a\in X$, $b\in H_e$ and $n\in\IN$ we should define the product
$a(bs^n)$. By the periodicity of the semigroup $Z(X)$, there is a
number $p\in\IN$ such that $f:=(ae)^p$ is an idempotent. If $fe<e$,
then we put $a(bs^n)=ab$. If $fe=e$, then $f=(ae)^p=(ae)^pe=fe=e$ and hence $ae$ belongs to the semigroup $\korin{\IN}{H_e}$. By Lemma~\ref{l:C-ideal}, $ae=aee\in\korin{\IN}{H_e}H_e\subseteq H_e$. So, we can put
$a(bs^n)=(ae)bs^n$. The choice of
$s\in\bigcap_{F\in[T]^{<\w}}\overline G_F$ guarantees that the
extended
binary operation is continuous. Now the density of $X$ in $Y$ implies that the extended operation is commutative and
associative. Since $Y\in\mathsf{T_{\!z}S}$, the semigroup $X$ is not $\mathsf{T_{\!z}S}$-closed, which is a desired contradiction completing the proof of Theorem~\ref{t:main}(4).

\section{Proof of Theorem~\ref{t:main}(5,6)}\label{s:5-6}

In this section we prove two lemmas implying statements (5) and (6) of Theorem~\ref{t:main}.

\begin{lemma}\label{l:group} Let $\mathsf i\in\{1,2,\mathsf z\}$. If $X$ is an ideally $\mathsf{T_{\!i}S}$-closed  semigroup, then for every viable idempotent $e\in V\!E(X)$ the maximal subgroup $H_e$ of $X$ is projectively $\mathsf{T_{\!i}S}$-closed and has bounded center.
\end{lemma}

\begin{proof} Assume that $X$ is an ideally $\mathsf{T_{\!i}S}$-closed semigroup and $e\in V\!E(X)$. To prove that the subgroup $H_e$ is projectively $\mathsf{T_{\!i}S}$-closed, take any homomorphism $h:H_e\to Y$ to a topological semigroup $Y\in\mathsf{T_{\!i}S}$ such that $h[H_e]$ is a discrete subgroup in $Y$. We need to prove that $h[H_e]$ is closed in $Y$. Replacing the topological semigroup $Y$ by the closure $\overline{h[H_e]}$ of $h[H_e]$ in $Y$, we can assume that $Y=\overline{h[H_e]}$.

Let us show that the complement $Y\setminus h[H_e]$ is an ideal in $Y$. In the opposite case we can find elements $y\in Y\setminus h[H_e]$ and $x\in Y$ such that $xy$ or $yx$ does not belong to $Y\setminus h[H_e]$. If $xy\notin Y\setminus h[H_e]$, then $xy\in h[H_e]$. Since $h[H_e]$ is a discrete subgroup of $Y$, there exists a neighborhood $O_{xy}$ of $xy$ in $Y$ such that $O_{xy}\cap h[H_e]=\{xy\}$. By the continuity of the semigroup operation in $Y$, there exist neighborhoods $O_x$ and $O_y$ of $x$ and $y$  in $Y$ such that $O_xO_y\subseteq O_{xy}$. Since $h[H_e]$ is dense in $Y$, we can choose an element $z\in O_x\cap h[H_e]$ and conclude that $z(O_y\cap h[H_e])\subseteq h[H_e]\cap O_{xy}=\{xy\}$ and hence $O_{y}\cap h[H_e]\subseteq \{z^{-1}xy\}$, which is not possible as $y\in Y\setminus h[H_e]$ is an accumulation point of the set $h[H_e]$ in $Y$. By analogy we can derive a contradiction assuming that $yx\notin Y\setminus h[H_e]$.

 Let $Y^0=Y\cup\{0\}$ be the $0$-extension of $Y$ by an external isolated zero. Since the idempotent $e$ is viable the set $I\defeq X\setminus\frac{H_e}e$ is an ideal in $X$. So we can consider the quotient semigroup $X/I$, and the homomorphism $h_e:X/I\to Y^0$,  defined by
$$h_e(x)=\begin{cases}h(xe)&\mbox{if $x\in\tfrac{H_e}e$};\\
0&\mbox{otherwise}.
\end{cases}
$$
The definition of the set $\frac{H_e}e=\{x\in X:xe=ex\in H_e\}$ guarantees that the homomorphism $h_e$ is well-defined.

Now consider the semigroup $Z=(X/I)\cup (Y\setminus h[H_e])$ in which $Y\setminus h[H_e]$ and $X/I$ are subsemigroups and for any $x\in X/I$ and  $y\in Y\setminus h[H_e]$ the products $xy$ and $yx$ are defined as $h_e(x)y$ and $yh_e(x)$, respectively. Endow the semigroup $Z$ with the topology $\tau$ consisting of the sets $W\subseteq Z$ such that for every $y\in W\cap (Y\setminus h[H_e])$ there exists a neighborhood $O_y$ of $y$ in $Y$ such that $(O_y\setminus h[H_e])\cup h_e^{-1}[O_y]\subseteq W$. It can be shown that $Z$ is a topological semigroup in the class $\mathsf{T_{\!i}S}$ containing $X/I$ as a discrete  subsemigroup. Since the semigroup $X$ is ideally $\mathsf{T_{\!i}S}$-closed, the semigroup $X/I$ is closed in $Z$, which implies that the set $h[H_e]$ is closed in $Y$.

Therefore, the maximal group $H_e$ is projectively $\mathsf{T_{\!i}S}$-closed. By Lemma~\ref{l:groups}(2), the center $Z(H_e)$ of the group $H_e$ is bounded.
\end{proof}

We say that a class $\C$ of topological semigroups is {\em closed under $0$-extensions} if for any topological semigroup $Y\in\C$ its $0$-extension $Y^0=Y\cup\{0\}$ belongs to the class $\C$. It is clear that for every $\mathsf i\in\{1,2,\mathsf z\}$ the class $\mathsf{T_{\!i}S}$ is closed under $0$-extensions.

\begin{lemma}\label{l:aC=>mg} Let $\C$ be a class of topological semigroups such that $\mathsf{T_{\!z}S}\subseteq\C\subseteq\mathsf{T_{\!1}S}$ and $\C$ is closed under $0$-extensions. If a semigroup $X$ is absolutely $\C$-closed, then for every viable idempotent $e\in V\!E(X)$, the maximal subgroup $H_e$ of $X$ is absolutely $\C$-closed and has finite center $Z(H_e)$.
\end{lemma}

\begin{proof} Since the idempotent $e$ is viable, the set $I\defeq X\setminus \frac{H_e}e$ is an ideal in $X$. By the definition of the set $\frac{H_e}e\defeq\{x\in X:xe=ex\in H_e\}$, the map
$h_e:X\to H_e^0$,
$$h_e(x)=\begin{cases}xe&\mbox{if $x\in\tfrac{H_e}e$};\\
0&\mbox{otherwise};
\end{cases}
$$is a well-defined homomorphism.
If $h_e[X]=H_e$, then the absolute $\C$-closedness of $X$ implies the absolute $\C$-closedness of the maximal subgroup $H_e$ and we are done. If $h_e[X]=H_e^0$, then the absolute $\C$-closedness of $X$ implies the absolute $\C$-closedness of the semigroup $H_e^0$. To prove that the group $H_e$ is absolutely $\C$-closed, take any homomorphism $f:H_e\to Y$ to a topological semigroup $Y\in\C$. Extend the homomorphism $f$ to the homomorphism $f^0:H_e^0\to Y^0$ such that $f(0)=0$. By our assumption, the class $\C$ is closed under $0$-extensions and hence contains the topological semigroup $Y^0$. Since the semigroup $H_e^0$ is absolutely $\C$-closed, the image $f^0[H_e^0]$ is closed in $Y^0$. Since the set $Y$ is closed in $Y^0$, the image $f[H_e]=Y\cap f^0[H_e^0]$ is closed in $Y$, witnessing that the group $H_e$ is absolutely $\C$-closed.
 By Lemma~\ref{l:groups}(1), the center $Z(H_e)$ of the group $H_e$ is finite.
\end{proof}

\section{Proof of Theorem~\ref{t:main}(7)}\label{s:7}

 In this section we prove three lemmas implying the statement 7 of Theorem~\ref{t:main}.

Recall that for a subset $A$ of a semigroup $X$ by $\korin{\IN}{\!A}$ we denote the set $\{x\in X:\exists n\in\IN\;\;x^n\in A\}$.

\begin{lemma}\label{l:root1}
If $X$ is an ideally $\mathsf{T_{\!z}S}$-closed semigroup, then the set $$B=\{e\in V\!E(X):(\korin{\IN}{\!H_e}\cap Z(X))\setminus H_e\neq\emptyset\}$$ is finite.
\end{lemma}

\begin{proof} To derive a contradiction, assume that the set $B$ is infinite. By Theorem~\ref{t:center}, the semigroup $Z(X)$ is chain-finite, periodic and nonsingular.
Given any idempotent $e\in B$, choose an element $x\in \korin{\IN}{\!H_e}\cap Z(X)$. By the periodicity of $Z(X)$, $x^n=e$ for some $n\in\IN$. It follows from $x\in Z(X)$ that $e\in Z(X)$. Therefore, $B\subseteq V\!E(X)\cap Z(X)$. Since the semigroup $Z(X)$ is chain-finite, we can apply the Ramsey Theorem \cite[Theorem 5]{Ramsey} and find an infinite subset $C\subseteq B$ such that $yx=xy\notin\{x,y\}$ for any distinct elements $x,y\in C$.

For every $e\in E(X)$ let ${\downarrow}^{\!\circ}e=\{x\in E(X):x<e\}$. By the definition of a viable idempotent, for every $e\in V\!E(X)$ the set $X\setminus\frac{H_e}e$ is an ideal in $X$.
Then for every $e\in C$ the intersection $$I_e\defeq \bigcap_{c\in C\setminus{\downarrow}^{\!\circ}e}(X\setminus\tfrac{H_c}c)$$is an ideal in $X$. It follows from $(H_e\cdot\frac{H_e}e)\cup(\frac{H_e}e\cdot H_e)\subseteq H_e$ that the union $J_e\defeq I_e\cup H_e$ is an ideal in $X$ and hence the union
$$J\defeq \bigcup_{e\in C}J_e$$ is an ideal in $X$.

\begin{claim}\label{cl:ae} For any $e\in C$ the set $Z(X)\cap\frac{H_e}e\setminus H_e$ contains an element $a_e$ such that $a_e^2\in H_e$.
\end{claim}

\begin{proof} Since $e\in C\subseteq B$, there exists $x\in (\korin{\IN}{H_e}\cap Z(X))\setminus H_e$. If $x^2\in H_e$, then put $a_e=x$. Otherwise, by the periodicity of $Z(X)$, there exists $n\in\mathbb N$ such that $x^n\in H_e$. Let $m\in\IN$ be the smallest number such that $x^m\in H_e$. Note that $2<m\le n$ and $x^{m-1}\in (\korin{\IN}{\!H_e}\cap Z(X))\setminus H_e$. Let $a_e=x^{m-1}$ and observe that $$a_e^2=x^{2m-2}=x^mx^{m-2}\in H_e\korin{\IN}{\!H_e}\subseteq H_e,$$by Lemma~\ref{l:C-ideal}. Also $a_ee\in \korin{\IN}{\!H_e}H_e\subseteq H_e$ and hence $a_e\in \frac{H_e}e\setminus H_e$.
\end{proof}

Consider the set $A\defeq\{a_e:e\in C\}$.

\begin{claim}\label{cl:AJ} $A\subseteq X\setminus J$.
\end{claim}

\begin{proof}
Assuming that $A\not\subseteq X\setminus J$, we can find two idempotents $e,c\in C$ such that $a_e\in J_c$.
We claim that $e\ne c$. In the opposite case, $a_e\notin H_e$ implies $a_e\in J_c\setminus H_e=J_e\setminus H_e\subseteq I_e\subseteq X\setminus\frac{H_e}e$, which contradicts the choice of $a_e$ in Claim~\ref{cl:ae}. This contradiction shows that $a\ne c$.

Since $a_e\in J_c$, we have $a_e^2\in J_c$ and $e=a_e^2(a_e^2)^{-1}\in J_c=I_c\cup H_c$. If follows from $e\ne c$ that $e\in J_c\setminus H_c\subseteq I_c$. Since $ce=ec\notin\{e,c\}$, we obtain that $e\in C\setminus{\downarrow}c$ and hence $e\in I_c\subseteq X\setminus \frac{H_e}e$, which contradicts the obvious inclusion $e\in\frac{H_e}e$. This contradiction completes the proof of $A\subseteq X\setminus J$.
\end{proof}

\begin{claim}\label{cl:AAJ} For any distinct idempotents $e,c\in C$ we have $a_ea_c\in I_e\cap I_c$ and $a_e\ne a_c$.
\end{claim}

\begin{proof} Assuming that $a_ea_c\notin I_e$, we can find an idempotent $z\in C\setminus {\downarrow}^\circ e$ such that $a_ea_c\in\frac{H_z}z$. Since $X\setminus \frac{H_z}z$ is an ideal, $a_ea_c\in\frac{H_z}z$ implies $a_e,a_c\in \frac{H_z}z$ and hence $za_e,za_c\in H_z$. Since the semigroup $X$ is periodic, there exists $n\in\IN$ such that $a_e^n=e$ and $a_c^n=c$. Then $ze=za_e^n=(za_e)^n\in H_{z}$ and $zc=zc^n\in H_z$. Taking into account that $ze,zc\in H_{z}$ are idempotents, we conclude that $ze=z=zc$ and hence $z\le c$ and $z\le e$. The choice of the set $C$ ensures that $c=z=e$, which contradicts the choice of $e\ne c$. This contradiction shows that $a_ea_c\in I_e$. By analogy we can show that $a_ea_c\in I_c$. Assuming that $a_e=a_c$, we obtain that $a_ea_c=a_e^2\in H_e\subseteq \frac{H_e}e$, which contradicts $a_ea_c\in I_e\subseteq X\setminus\frac{H_e}e$.
\end{proof}

Claims~\ref{cl:AJ} and \ref{cl:AAJ} imply that $A$ is an infinite subset of $X\setminus J$ such that $$AA\subseteq J\cup\{a_e^2:e\in C\}\subseteq J\cup\bigcup_{c\in C}H_e=J.$$

Since $X$ is ideally $\mathsf{T_{\!z}S}$-closed, the quotient semigroup $X/J$ is $\mathsf{T_{\!z}S}$-closed. However, the set $A\subseteq Z(X)\setminus J$ is a witness to the singularity of $Z(X/J)$ which contradicts Theorem~\ref{t:center}. The obtained contradiction implies that the set $B$ is finite.
\end{proof}

The following lemma has been proved in \cite[Lemma 7.5]{BB}.

\begin{lemma}\label{l:root2} Let $X$ be an ideally $\mathsf{T_{\!z}S}$-closed semigroup such that for some $e\in E(X)\cap Z(X)$ the semigroup $H_e\cap Z(X)$ is bounded. Then the set $(\!\sqrt[\infty]{\!H_e}\cap Z(X))\setminus H_e$ is finite.
\end{lemma}

The next lemma proves  the last statement of Theorem~\ref{t:main}.

\begin{lemma} If $X$ is an ideally $\mathsf{T_{\!z}S}$-closed  semigroup, then the set $Z(X)\cap\korin{\IN}{V\!E(X)}\setminus H(X)$ is finite.
\end{lemma}

\begin{proof} By Theorems~\ref{t:center} and \ref{t:main}(3), the semigroup $Z(X)$ is periodic and group-bounded. By Lemma~\ref{l:root1}, the set $B=\{e\in V\!E(X):Z(X)\cap\korin{\IN}{H_e}\setminus H_e\ne \emptyset\}$ is finite. By Lemma~\ref{l:root2}, for every $e\in B$ the set $Z(X)\cap\korin{\IN}{H_e}\setminus H_e$ is finite. Now we see that the set $$Z(X)\cap\korin{\IN}{V\!E(X)}\setminus H(X)\subseteq\bigcup_{e\in B}Z(X)\cap\korin{\IN}{H_e}\setminus H_e$$ is finite.
\end{proof}

\end{document}